\documentclass[]{theclass}
\usepackage{graphicx, xfrac, lineno, float, subcaption, tasks, comment, xcolor, booktabs, multirow}
\usepackage[normalem]{ulem}
\usepackage{datetime}
\usepackage{colortbl}
\usepackage{enumerate}

\settasks{
	counter-format=(tsk[r]),
	label-width=4ex
}

\begin{document}


\begin{frontmatter}

\titledata{On two graph isomorphism problems}{}           

\authordata{John Baptist Gauci}
{Department of Mathematics, University of Malta, Malta}
{john-baptist.gauci@um.edu.mt}
{}

\authordata{Jean Paul Zerafa}
{Department of Technology and Entrepreneurship Education\\University of Malta, Malta; \\
Department of Computer Science, Comenius University\\ 842 48 Bratislava, Slovakia}
{zerafa.jp@gmail.com}
{}


\begin{abstract}
In 2015, Bogdanowicz gave a necessary and sufficient condition for a 4-regular circulant graph to be isomorphic to the Cartesian product of two cycles. Accordion graphs, denoted by $A[n,k]$, are 4-regular graphs on two parameters $n$ and $k$ which were recently introduced by the authors and studied with regards to Hamiltonicity and matchings. These graphs can be obtained by a slight modification in some of the edges of the Cartesian product of two cycles. Motivated by the work of Bogdanowicz, the authors also determined for which values of $n$ and $k$ the accordion graph $A[n,k]$ is circulant. In this work we investigate what parameters a 4-regular circulant graph must have in order to be isomorphic to an accordion graph, thus providing a complete characterisation similar to that given by Bogdanowicz. We also give a necessary and sufficient condition for two accordion graphs with distinct parameters to be isomorphic.
\end{abstract}

\end{frontmatter}

\section{Introduction}

Graph isomorphism problems are considered to be notoriously hard and one of the most prominent and still unsolved problems in computer science and mathematics is the so-known ``Graph Isomorphism Problem", that is, the computational problem of determining in polynomial time whether two finite graphs are isomorphic. In particular, Aurenhammer \emph{et al.} \cite{ahi1992} proved in 1992 that there is an efficient polynomial time algorithm to resolve the isomorphism problem for the Cartesian product of graphs. Some time later, in 2004, Muzychuk \cite{muzychuk} showed that determining whether two circulant graphs are isomorphic or not can be solved in polynomial time.

Let $G$ be a simple graph with vertex set $V(G)$ and edge set $E(G)$. The graph $G$ is said to be \emph{$4$-regular} or \emph{quartic} if every vertex in $G$ has exactly four edges incident to it. The \emph{Cartesian product} $G\square H$ of two graphs $G$ and $H$ is a graph whose vertex set is the Cartesian product $V(G) \times V(H)$ of $V(G)$ and $V(H)$. Two vertices $(w_i,z_j)$ and $(w_k,z_{\ell})$ are adjacent precisely if  $w_i=w_k$ and $z_jz_{\ell}\in E(H)$ or $w_iw_k \in E(G)$ and $z_j=z_{\ell}$. The graph $G\square H$ can be partitioned into $|V(H)|$ disjoint copies of $G$ or into $|V(G)|$ disjoint copies of $H$. We shall refer to each such copy of $G$ or $H$ as a {\em canonical} copy, particularly to differentiate from other possible copies of $G$ or $H$ that $G\square H$ might contain.

For distinct positive integers $a$ and $b$, $\textrm{Ci}[2n,\{a,b\}]$ denotes the \emph{quartic circulant graph} on the vertices $\{x_{i}:i\in[2n]\}$, such that $x_i$ is adjacent to the vertices in the set $\{x_{i+a},x_{i-a}, \linebreak x_{i+b}, x_{i-b}\}$. Here and henceforth, operations (including addition and subtraction) in the indices of the vertices $x_{i}$ of $\textrm{Ci}[2n,\{a,b\}]$ are taken modulo $2n$, with complete residue system $\{1,\ldots, 2n\}$. We say that the edges arising from these adjacencies have length $a$ and $b$, accordingly. Furthermore, since operations in the indices of the vertices $x_{i}$ are taken modulo $2n$, we can further assume that both $a$ and $b$ are strictly less than $n$. When a 4-regular graph $G$ is isomorphic to a quartic circulant graph we shall say that $G$ is \emph{circulant}. For $t\geq 3$, a \emph{cycle} of length $t$ (or a $t$-cycle), denoted by $C_{t}=(v_{1}, \ldots, v_{t})$, is a sequence of distinct vertices $v_1,v_2,\ldots,v_{t}$ with corresponding edge set $\{v_{1}v_{2}, \ldots, v_{t-1}v_{t}, v_{t}v_{1}\}$. For definitions not explicitly stated here we refer the reader to \cite{Diestel}.

Accordion graphs (introduced in \cite{accordionspmh}) can be obtained from the Cartesian product of two cycles as follows. Consider, for example, $C_{4}\square C_{5}$, and let $(\ell_{1}, \ell_{2}, \ell_{3},\ell_{4})$ and $(r_{1},r_{2},r_{3},r_{4})$ be two canonical $4$-cycles of $C_{4}\square C_{5}$ such that $\ell_{i}$ is adjacent to $r_{i}$, for every $i\in[4]$. The resulting graph after deleting the edges $\{\ell_{i}r_{i}:i\in[4]\}$ from $C_{4}\square C_{5}$ is isomorphic to $C_{4}\square P_{5}$, where $P_{5}$ is the path on five vertices of length $4$. If we add the edges $\{r_{1}\ell_{3},r_{2}\ell_{4},r_{3}\ell_{1},r_{4}\ell_{2},\}$ to $C_{4}\square P_{5}$, the resulting graph is, in fact, the accordion graph $A[10,5]$ (depicted in Figure \ref{Figure AccCart}) as defined in the following definition.

\begin{figure}[h]
\centering
\includegraphics[width=.59\textwidth]{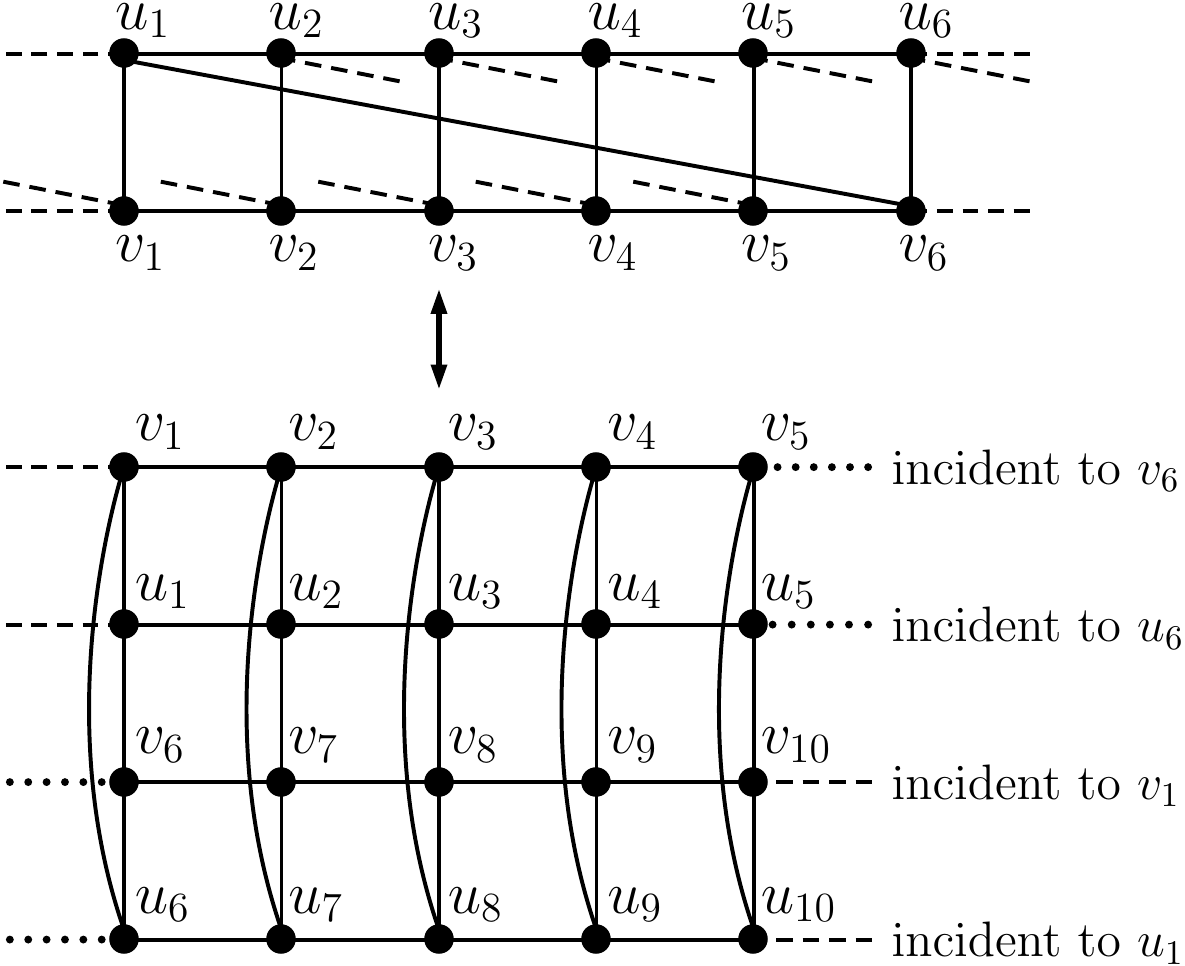}
\caption{Two different drawings of $A[10,5]$}
\label{Figure AccCart}
\end{figure}

\begin{definition} \label{def AccordionGraphs}
Let $n$ and $k$ be integers such that  $n\geq 3$ and $1\leq k\le\frac{n}{2}$. The \emph{accordion graph} $A[n,k]$  is the quartic graph with vertices $\{u_1, u_2,\ldots , u_{n}, v_1, v_2,\ldots , v_{n}\}$ such that:
\begin{enumerate}[(i)]
\item $(u_{1},u_{2},\ldots, u_{n})$ and $(v_{1},v_{2},\ldots, v_{n})$ are two $n$-cycles;
\item $u_{i}$ is adjacent to $v_{i}$, for every $i\in[n]$; and
\item $u_{i}$ is adjacent to $v_{i+k}$, for every $i\in[n]$, where addition in the subscripts is taken modulo $n$, with complete residue system $\{1,\ldots, n\}$.
\end{enumerate} 
The edges $u_{i}u_{i+1}$ and $v_{i}v_{i+1}$ are called the \emph{outer-cycle edges} and the \emph{inner-cycle edges}, respectively, or simply the \emph{cycle edges}, collectively; and the edges $u_iv_i$ and $u_iv_{i+k}$ are called the \emph{vertical spokes} and the \emph{diagonal spokes}, respectively, or simply the \emph{spokes}, collectively. We denote the set $\{u_{1},u_{2},\ldots, u_{n}\}$ of vertices on the outer-cycle of $A[n,k]$ by $\mathcal{U}$, and the set $\{v_{1},v_{2},\ldots, v_{n}\}$ of vertices on the inner-cycle of $A[n,k]$ by $\mathcal{V}$. We remark that, here and henceforth, operations (including addition and subtraction) in the indices of the vertices $u_{i}$ and $v_{i}$ of $A[n,k]$ are taken modulo $n$, with complete residue system $\{1,\ldots, n\}$.
\end{definition}

As can be seen in the following remark, the greatest common divisor of $n$ and $k$, denoted by $\gcd(n,k)$, will prove to be useful in the study of accordion graphs.

\begin{remark}\label{RemarkDrawing} 
The graph obtained from $A[n,k]$ after deleting the edges 
\begin{linenomath}
$$\{u_{tn_{2}}u_{tn_{2}+1}, v_{tn_{2}}v_{tn_{2}+1} : n_{2}=\gcd(n,k) \textrm{ and } t\in\{1,\ldots,\tfrac{n}{n_{2}}\}\}$$ 
\end{linenomath}
is isomorphic to the Cartesian product $C_{n_{1}} \square P_{n_{2}}$, where $n_{1}n_{2}=2n$. On the other hand, consider $C_{n_{1}}\square P_{n_{2}}$, for some even $n_{1}\geq 4$, and for some integer $n_{2}$. Note that $n_{1}$ was taken to be even because accordion graphs are of even order. Assume first that $n_{2}>1$. Let $(\ell_{1}, \ell_{2}, \ldots, \ell_{n_{1}})$ and $(r_{1},r_{2},\ldots,r_{n_{1}})$ be the two canonical $n_{1}$-cycles of $C_{n_{1}}\square P_{n_{2}}$ such that for every $i\in[n_{1}]$, $\ell_{i}$ and $r_{i}$ have degree $3$ and are endvertices of the $i$\textsuperscript{th} canonical copy of $P_{n_{2}}$ in $C_{n_{1}}\square P_{n_{2}}$. Depending on $n_{1}$, there could be various distinct accordion graphs $A[n,k]$ which can be obtained from $C_{n_{1}}\square P_{n_{2}}$ by adding edges to the vertices having degree $3$. 
For such an accordion graph $A[n,k]$, $n$ is equal to $\frac{n_{1}n_{2}}{2}$. Moreover, $k$ can be any integer in $\{1,\ldots,\frac{n}{2}\}$ such that $\gcd(n,k)=n_{2}$. Let $k$ be such an integer. The set consisting of the edges that need to be added to $C_{n_{1}}\square P_{n_{2}}$ such that the resulting graph is isomorphic to $A[n,k]$ is $\{r_{i}\ell_{i+2\gamma\pmod{n_{1}}}:i\in[n_{1}]\}$, where $\gamma$ is the least positive integer such that $\gamma k\equiv n_{2}\pmod{n}$. When $n_{2}=1$, $C_{n_{1}}\square P_{n_{2}}$ consists of the (Hamiltonian) cycle $(w_{1},w_{2}, \ldots, w_{n_{1}})$, and in order to obtain an accordion graph $A[n,k]$ by adding edges to $C_{n_{1}}\square P_{1}$, $n$ must be equal to $\frac{n_{1}}{2}$ and $k$ can be any integer in $\{1,\ldots,\frac{n}{2}\}$ such that $\gcd(n,k)=1$, that is, $n$ and $k$ are coprime. Let $k$ be such an integer. The set consisting of the edges that need to be added to $C_{n_{1}}\square P_{1}$ such that the resulting graph is isomorphic to $A[n,k]$ is $\{w_{i}w_{i+2\gamma\pmod{n_{1}}}:i\in[n_{1}]\}$, where $\gamma$ is the least positive integer such that $\gamma k\equiv 1\pmod{n}$.
\end{remark}

In Section \ref{section isomorphic accordions} we consider isomorphisms in the class of accordion graphs and we give a complete solution to the problem of determining whether two accordion graphs are isomorphic or not. Clearly, if $A[n,k_{1}]$ and $A[n',k_{2}]$ are isomorphic, $n$ must be equal to $n'$, since $|V(A[n,k_{1}])|=|V(A[n',k_{2}])|$. Consequently, in Theorem \ref{theorem iso accordions}, we give the conditions on $k_{1}$ and $k_{2}$ for which the two accordion graphs $A[n,k_{1}]$ and $A[n,k_{2}]$ are isomorphic.
In Section \ref{section circulant accordion graphs}, we then tackle and completely solve the problem of when a quartic circulant graph is isomorphic to an accordion graph. This was motivated by the work of Bogdanowicz \cite{bogdanowicz} where he proved the following theorem dealing with quartic circulant graphs and the Cartesian product of two cycles.

\begin{theorem}\cite{bogdanowicz}
The quartic circulant graph $\textrm{Ci}[n',\{a_{1},a_{2}\}]$ is isomorphic to $C_{n_{1}}\square C_{n_{2}}$ if and only if the following two conditions hold:
\begin{enumerate}[(i)]
\item $n'=n_1n_2$;
\item $n_1=\gcd(n',a_j)$ and $n_2=\gcd(n',a_{3-j})$, where $j=1$ or $j=2$; and
\item $\gcd(n_1,n_2)=1$.
\end{enumerate}
\end{theorem}

Here, we obtain an analogous result for accordion graphs (see Section \ref{section circulant accordion graphs}). This characterisation shall be done in two steps: the bipartite case (Theorem \ref{theorem bip iso}) and the non-bipartite case (Theorem \ref{theorem nonbip iso}), where we give the necessary and sufficient conditions for a quartic circulant graph $\textrm{Ci}[2n,\{a,b\}]$ to be isomorphic to the accordion graph $A[n,k]$.

\section{Main results}

Before proceeding, we give some results from \cite{accordionspmh} which shall be used to prove our main results.

\begin{lemma}[\cite{accordionspmh}]\label{lemma bipartite}
The accordion graph $A[n,k]$ is bipartite if and only if both $n$ and $k$ are even.
\end{lemma}

\begin{theorem}[\cite{accordionspmh}]\label{theorem accordion char circ}
The accordion graph $A[n,k]$ is circulant if and only if exactly one of the following occurs:
\begin{enumerate}[(i)]
\item $k$ is odd; or
\item $k$ is even and $n$ is odd; or
\item $k=2$ and $n$ is even.
\end{enumerate}
\end{theorem}

In particular, by Lemma \ref{lemma bipartite} and Theorem \ref{theorem accordion char circ}, the bipartite accordion graph $A[n,k]$ is circulant only when $k=2$.
We also remark that every accordion graph $A[n,k]$ with outer-cycle $(u_1, u_2, \ldots, u_n)$ and inner cycle $(v_1, v_2, \ldots, v_n)$ has the following natural automorphism $\psi:V(A[n,k])\rightarrow V(A[n,k])$ which exchanges the outer- and inner-cycles as follows:
\begin{itemize}
\item $\psi(u_{1})=v_{1}$ and $\psi(v_{1})=u_{1}$; and
\item for every $i\in\{2,3,\ldots, n\}$, $\psi(u_{i})=v_{2-i}$ and $\psi(v_{i})=u_{2-i}$.
\end{itemize}
This means that the $n$-tuple of vertices $(u_{1}, u_{2}, \ldots, u_{n})$ is  mapped to the $n$-tuple of vertices $(v_{1}, v_{n}, \ldots, v_{2})$, and vice-versa.

\subsection{Isomorphisms in the class of accordion graphs}\label{section isomorphic accordions}

In the following theorem we establish when two accordion graphs $A[n,k_1]$ and $A[n,k_2]$, $k_1\neq k_2$, are isomorphic.

\begin{theorem}\label{theorem iso accordions}
Let $k_{1}$ and $k_{2}$ be two integers such that $1\leq k_{1}<k_{2}\leq\frac{n}{2}$. The accordion graphs $A[n,k_{1}]$ and $A[n,k_{2}]$ are isomorphic if and only if the following two conditions hold:
\begin{enumerate}[(i)]
\item $\gcd(n,k_{1})=\gcd(n,k_{2})=2$; and 
\item $\frac{k_{1}k_{2}}{2}\equiv\pm2\pmod{n}$. 
\end{enumerate}
\end{theorem}

\begin{proof}
($\Rightarrow$) Assume first that $A[n,k_1]$ and $A[n,k_2]$ are isomorphic, where $1\leq k_1<k_2 \leq \frac{n}{2}$. Let the outer- and inner-cycles of $A[n,k_{1}]$ be  $(u_{1},\ldots, u_{n})$ and $(v_{1},\ldots, v_{n})$, respectively, such that, for every $i\in[n]$, $u_{i}v_{i}$ and $u_{i}v_{i+k_{1}}$ are edges of $A[n,k_{1}]$. 
Similarly, let the outer- and inner-cycles of $A[n,k_{2}]$ be $(u_{1}',\ldots, u_{n}')$ and $(v_{1}',\ldots, v_{n}')$, respectively, such that, for every $i\in[n]$, $u_{i}'v_{i}'$ and $u_{i}'v_{i+k_{2}}'$ are edges of $A[n,k_{2}]$. 
Since $k_{1}\neq k_{2}$, there exists an isomorphism $\pi:V(A[n,k_{1}])\rightarrow V(A[n,k_{2}])$ which maps the two endvertices of a cycle edge of $A[n,k_{1}]$ to the two endvertices of a spoke of $A[n,k_{2}]$, and vice-versa. 
Due to the natural automorphism $\psi$ of $A[n,k_{2}]$ mentioned above, without loss of generality, we can assume that $\pi(u_{1})=v_{1}'$ and $\pi(u_{2})=u_{1}'$. We first show that $k_{1}$ cannot be equal to $1$ or $2$.\\

\noindent\textbf{Claim.} $k_{1}\geq 3$. 

\noindent{\emph{Proof of Claim.}} We first remark that for every $n\geq 3$, the accordion graph $A[n,k_{1}]$ contains a cycle of length $3$ if and only if $k_{1}=1$. Consequently, there does not exist $k_{2}>1$, such that $A[n,1]\cong A[n,k_{2}]$, because, in that case, $A[n,k_{2}]$ would be triangle-free. Therefore, $k_{1}\geq 2$. 
Suppose $k_{1}=2$. By Lemma \ref{lemma bipartite}, if $n$ is even, $k_{2}$ must be even, because $A[n,2]$ is bipartite, and so $A[n,k_{2}]$ must be bipartite as well. However, by Theorem \ref{theorem accordion char circ}, $A[n,2]$ is circulant whilst $A[n,k_{2}]$ is not, a contradiction. Thus, $n$ must be odd, and since $k_{1}=2$, $k_{2}\geq 3$ and $n\geq 2k_{2}+1\geq 7$. We consider three cases depending on whether $\pi^{-1}(u_{2}')$ is $u_3$, $v_2$ or $v_4$.\\

\noindent\textbf{Case 1.} If $\pi^{-1}(u_{2}')=u_{3}$, then $\pi^{-1}(v_{2}')$ is equal to $u_{4}, v_{5}$ or $v_{3}$. If we first suppose that $\pi^{-1}(v'_2)=u_4$, then, $u_1$ is adjacent to $u_4$, a contradiction, since $n\geq 7$. On the other hand, if we suppose that $\pi^{-1}(v'_2)=v_5$, then we get that $u_1v_5$ is an edge of $A[n,k_1]$, a contradiction since $k_1=2$. Thus, $\pi^{-1}(v'_2)=v_3$.
However, this means that $\pi^{-1}(v_{3}')$ is either equal to $v_{2}$ or $v_{4}$, and in both cases, $\pi(v_{2})$ and $\pi(v_{4})$ have to be adjacent to $\pi(u_{2})$, since $\pi$ is an isomorphism and $u_{2}v_{2}$ and $u_{2}v_{4}$ are edges in $A[n,k_{1}]$, when $k_{1}=2$. Therefore, $v_{3}'$ must be adjacent to $u_{1}'$, a contradiction, since $2=k_{1}<k_{2}\leq \frac{n}{2}$.\\

\noindent\textbf{Case 2.} If $\pi^{-1}(u_{2}')=v_{2}$, then $\pi^{-1}(u_{3}')$ is equal to $v_{1}, v_{3}$ or $u_{n}$. In each case, $\pi^{-1}(u_{3}')$ is adjacent to $u_{1}$ in $A[n,k_{1}]$, in particular, since $k_{1}=2$. This implies that $u_{3}'$ must be adjacent to $v_{1}'$ in $A[n,k_{2}]$, since $\pi$ is an isomorphism. However, if $k_{2}\equiv-2\pmod{n}$, then, $k_{2}=n-2$, and since by Definition \ref{def AccordionGraphs} $k_{2}\leq \frac{n}{2}$, this implies that $n\leq 4$, a contradiction.\\

\noindent\textbf{Case 3.} If $\pi^{-1}(u_{2}')=v_{4}$, then $\pi^{-1}(v_{2}')=v_{3}$, and consequently, $\pi^{-1}(v_{3}')$ is equal to $u_{3}$ or $v_{2}$, which are both adjacent to $u_{2}=\pi^{-1}(u_{1}')$ in $A[n,k_{1}]$. As before, this means that $u_{1}'$ has to be adjacent to $v_{3}'$, a contradiction once again, since $2=k_{1}<k_{2}\leq \frac{n}{2}$. Therefore, $k_{1}\neq 2$.\,\,\,{\tiny$\blacksquare$}\\

Hence, we can further assume that $k_{1}\geq 3$ and $n\geq 2k_{2}\geq 2(k_{1}+1)\geq 8$. We next show that $\gcd(n,k_{1})=\gcd(n,k_{2})=2$. The vertex $u_{2}'$ is adjacent to $u_{1}'$ and is the endvertex of a path of length two whose other endvertex is $v_{1}'$ and which does not contain $u_{1}'$. Thus, $\pi^{-1}(u'_2)$ is adjacent to $u_2$ and is also the endvertex of a path of length two not containing $u_{2}$ whose other endvertex is $u_{1}$. Consequently, since $k_1\geq 3$ and $n\geq 8$, $\pi^{-1}(u'_2) \neq u_3$, and we have two cases to consider, namely when $\pi^{-1}(u'_2)$ is either $v_2$ or $v_{2+k_1}$.\\

\noindent\textbf{Case (a).} Consider first $\pi^{-1}(u_{2}')=v_{2}$. Since $\pi^{-1}(v_{1}')=u_{1}$ and $\pi^{-1}(u_{1}')=u_{2}$, $\pi^{-1}(v_{2}')$ cannot be equal to $u_{n}$ or $v_{3}$, because $k_1\geq 3$ and thus, $u_{n}v_{2}$ and $u_{1}v_{3}$ are not edges in $A[n,k_{1}]$. Hence, $\pi^{-1}(v_{2}')=v_{1}$. Consequently, $\pi^{-1}(u_{3}')$ is either equal to $u_{2-k_{1}}$ or $v_{3}$. However, if $\pi^{-1}(u_{3}')=v_{3}$, then $\pi^{-1}(v_{3}')$ must be adjacent to both $v_1$ and $v_3$, which is impossible. Therefore, $\pi^{-1}(u_{3}')=u_{2-k_{1}}$, and as a result, $\pi^{-1}(v_{3}')$ must be adjacent to both $v_1$ and $u_{2-k_1}$, implying that $\pi^{-1}(v_{3}')=u_{1-k_{1}}$, since $k_{1}\geq 3$. This means that $u_{2}'u_{3}'$ and $v_{2}'v_{3}'$ are the images of (consecutive) diagonal spokes of $A[n,k_{1}]$ under $\pi$. Repeating the same reasoning used for $u_{2}'$ and $v_{2}'$, we have $\pi^{-1}(u_{4}')=v_{2-k_{1}}$ and $\pi^{-1}(v_{4}')=v_{1-k_{1}}$, that is, $u_{3}'u_{4}'$ and $v_{3}'v_{4}'$ are the images of (consecutive) vertical spokes of $A[n,k_{1}]$ under $\pi$. Continuing in this manner, we deduce that the spokes of $A[n,k_{1}]$ induce two vertex-disjoint $n$-cycles corresponding to $(u_{1}',\ldots, u_{n}')$ and $(v_{1}',\ldots, v_{n}')$. This means that $\gcd(n,k_{1})=2$. Thus, the spokes of $A[n,k_1]$ are mapped onto the cycle edges of $A[n,k_2]$, implying that the cycle edges of $A[n,k_1]$ are mapped onto the spokes of $A[n,k_2]$ and so, $\gcd(n,k_2)=2$. We note that this implicitly implies that $n$, $k_1$ and $k_2$ are all even.
Moreover, we note that the cycles $(u_{1}',\ldots, u_{n}')$ and $(v_{1}',\ldots, v_{n}')$ in $A[n,k_{2}]$ respectively correspond to the cycles 
\begin{linenomath}
$$(u_{2},v_{2},u_{2-k_{1}},v_{2-k_{1}},\ldots,u_{2-(\frac{n}{2}-1)k_{1}},v_{2-(\frac{n}{2}-1)k_{1}}),$$
\end{linenomath} 
and 
\begin{linenomath}
$$(u_{1},v_{1},u_{1-k_{1}},v_{1-k_{1}},\ldots,u_{1-(\frac{n}{2}-1)k_{1}},v_{1-(\frac{n}{2}-1)k_{1}}),$$
\end{linenomath} 
in $A[n,k_{1}]$, under the isomorphism $\pi$. In particular, $u_{1}'$, which is equal to $\pi(u_{2})$, is adjacent to $v_{1+k_{2}}'$, which is equal to $\pi(u_{1-(\frac{k_{2}}{2})k_{1}})$. This is possible since $k_{2}$ is even. Since $\pi^{-1}(v_{1}')=u_{1}$, the vertex $u_{1-(\frac{k_{2}}{2})k_{1}}$ must be equal to $u_{3}$, implying that $\frac{k_{1}k_{2}}{2}\equiv-2\pmod{n}$.\\

\noindent\textbf{Case (b).} Now, if $\pi^{-1}(u_{2}')=v_{2+k_{1}}$, one can similarly deduce that $\gcd(n,k_{1})=\linebreak \gcd(n,k_{2})=2$, and that the cycles $(u_{1}',\ldots, u_{n}')$ and $(v_{1}',\ldots, v_{n}')$ in $A[n,k_{2}]$ respectively correspond to the cycles 
\begin{linenomath}
$$(u_{2},v_{2+k_{1}},u_{2+k_{1}},v_{2+2k_{1}},\ldots,u_{2+(\frac{n}{2}-1)k_{1}},v_{2}),$$
\end{linenomath} 
and 
\begin{linenomath}
$$(u_{1},v_{1+k_{1}},u_{1+k_{1}},v_{1+2k_{1}},\ldots,u_{1+(\frac{n}{2}-1)k_{1}},v_{1}),$$
\end{linenomath} 
in $A[n,k_{1}]$, under the isomorphism $\pi$. In particular, $u_{1}'$, which is equal to $\pi(u_{2})$, is adjacent to $v_{1+k_{2}}'$, which is equal to $\pi(u_{1+(\frac{k_{2}}{2})k_{1}})$. Once again, this is possible since $k_{2}$ is even. Since $\pi^{-1}(v_{1}')=u_{1}$, the vertex $u_{1+(\frac{k_{2}}{2})k_{1}}$ is equal to $u_{3}$, implying that $\frac{k_{1}k_{2}}{2}\equiv2\pmod{n}$. Therefore, $\frac{k_{1}k_{2}}{2}\equiv \pm 2\pmod{n}$, as required.\\

($\Leftarrow$) Conversely, assume $\gcd(n,k_{1})=\gcd(n,k_{2})=2$ and $\frac{k_{1}k_{2}}{2}\equiv\pm2\pmod{n}$. Since $\gcd(n,k_{1})=2$, $A[n,k_{1}]$ contains the following two vertex-disjoint $n$-cycles: 
\begin{linenomath}
$$C_{1}=(v_{1},u_{1},v_{1+k_{1}},u_{1+k_{1}},\ldots,v_{1+(\frac{n}{2}-1)k_{1}},u_{1+(\frac{n}{2}-1)k_{1}}),$$
\end{linenomath} 
and 
\begin{linenomath}
$$C_{2}=(v_{2},u_{2},v_{2+k_{1}},u_{2+k_{1}},\ldots, v_{2+(\frac{n}{2}-1)k_{1}},u_{2+(\frac{n}{2}-1)k_{1}}).$$
\end{linenomath}
We consider two cases, depending on the value of $\frac{k_{1}k_{2}}{2}$.\\ 

\begin{figure}[h]
      \centering
      \includegraphics[width=0.3\textwidth]{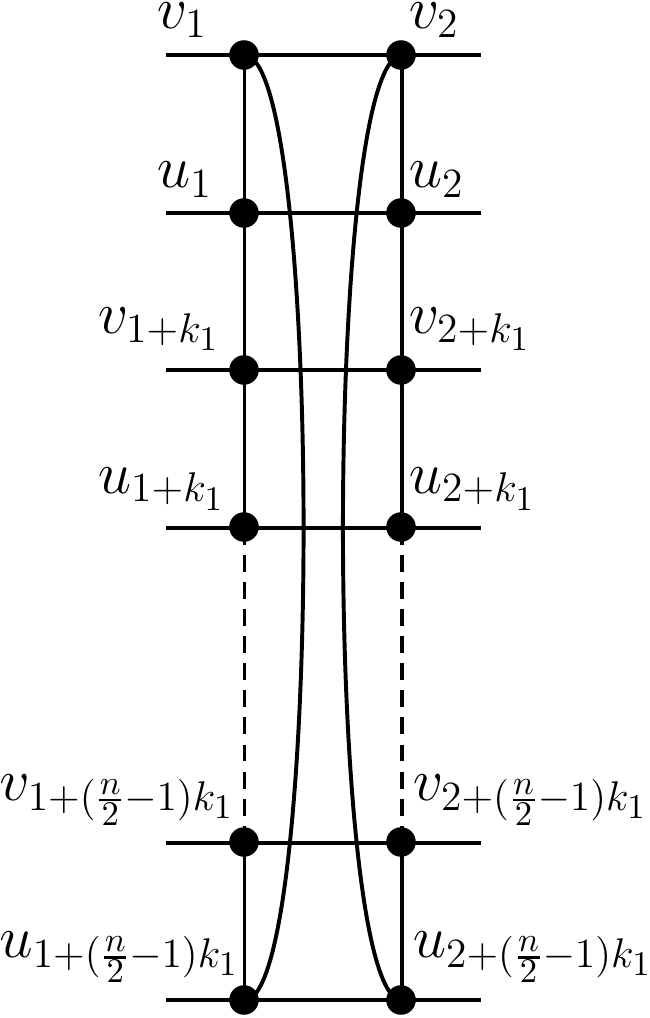}
      \caption{$A[n,k_{1}]$ when $\gcd(n,k_{1})=2$}
      \label{figure gcdnk=2}
\end{figure} 

\noindent\textbf{Case A.} Assume first that $\frac{k_{1}k_{2}}{2}\equiv-2\pmod{n}$. We label the consecutive vertices of $C_{1}$ and $C_{2}$ as $u_{1}',\ldots, u_{n}'$ and $v_{1}',\ldots, v_{n}'$, respectively, such that $u_{1}'=v_{1}, u_{2}'=u_{1}, v_{1}'=v_{2}$, and $v_{2}'=u_{2}$. Let $i\in[n]$. We claim that $u_{i}'u_{i+1}', v_{i}'v_{i+1}', u_i'v_i',u_i'v_{i+k_{2}}'$ are edges of $A[n,k_{1}]$. By definition of $C_{1}$ and $C_{2}$, it can be seen that $u_{i}'u_{i+1}', v_{i}'v_{i+1}', u_i'v_i'$ are edges of $A[n,k_{1}]$, and so, all that is left to show is that $u_i'v_{i+k_{2}}'\in E(A[n,k_{1}])$. Without loss of generality, assume $i=1$. Since $\frac{k_{1}k_{2}}{2}\equiv-2\pmod{n}$, $v_{1+k_{2}}'$ is equal to $v_{2+(\frac{k_{2}}{2})k_{1}}=v_{n}$, which is in fact adjacent to $u_{1}'$, since $u_{1}'=v_{1}$. By a similar argument, $u_i'v_{i+k_{2}}'\in E(A[n,k_{1}])$, for every $i\in[n]$, proving our claim. Consequently, the quartic graph with vertices $\{u_1', u_2',\ldots , u_{n}', v_1', v_2',\ldots , v_{n}'\}$ and edge set $\{u_{i}'u_{i+1}',v_{i}'v_{i+1}',u_i'v_i',u_i'v_{i+k_{2}}': i\in[n]\},$ is a spanning subgraph of $A[n,k_{1}]$ which is isomorphic to $A[n,k_{2}]$. Since $A[n,k_{1}]$ is 4-regular, $A[n,k_{1}]\cong A[n,k_{2}]$, as required.\\

\noindent\textbf{Case B.} If $\frac{k_{1}k_{2}}{2}\equiv2\pmod{n}$, we label the consecutive vertices of $C_{1}$ and $C_{2}$ as \linebreak$u_{1}',u_{n}',\ldots, u_{2}'$ and $v_{1}',v_{n}',\ldots, v_{2}'$, respectively, such that $u_{1}'=v_{1}, u_{n}'=u_{1}, v_{1}'=v_{2}$, and $v_{n}'=u_{2}$. Once again, we claim that, for every $i\in[n]$, $u_{i}'u_{i+1}', v_{i}'v_{i+1}', u_i'v_i',u_i'v_{i+k_{2}}'$ are edges of $A[n,k_{1}]$. As in Case A, it can be seen that $u_{i}'u_{i+1}', v_{i}'v_{i+1}', u_i'v_i'$ are edges of $A[n,k_{1}]$, and once again, all that is left to show is that $u_i'v_{i+k_{2}}'\in E(A[n,k_{1}])$. Without loss of generality, assume $i=1$. Since $\frac{k_{1}k_{2}}{2}\equiv2\pmod{n}$, the vertex $v_{1+k_{2}}'$ is equal to $v_{2-(\frac{k_{2}}{2})k_{1}}=v_{n}$, which is in fact adjacent to $u_{1}'$, since $u_{1}'=v_{1}$. This proves our claim, and as in the previous case, we have $A[n,k_{1}]\cong A[n,k_{2}]$, as required.
\end{proof}

\begin{corollary}
Let $k_{1}$, $k_{2}$ and $k_{3}$ be three positive integers such that $k_{1}\neq k_{2}$ and $k_{1}\neq k_{3}$. If $A[n,k_{1}]\cong A[n,k_{2}]$ and $A[n,k_{1}]\cong A[n,k_{3}]$, then $k_{2}=k_{3}$. 
\end{corollary}

\begin{proof}
Suppose $k_{2}\neq k_{3}$, and without loss of generality, assume that $k_{2}>k_{3}$. By Theorem \ref{theorem iso accordions}, $\frac{k_{1}k_{2}}{2}=\alpha n \pm 2$, and $\frac{k_{1}k_{3}}{2}=\beta n \pm 2$, for some positive integers $\alpha$ and $\beta$. Consequently, we either have $\frac{k_{1}(k_{2}+k_{3})}{2}=(\alpha+\beta)n$, or $\frac{k_{1}(k_{2}-k_{3})}{2}=(\alpha-\beta)n$. In particular, we note that by Theorem \ref{theorem iso accordions}, both $k_{2}$ and $k_{3}$ are even, and that by Definition \ref{def AccordionGraphs}, $\frac{(k_{2}\pm k_{3})}{2}<\frac{n}{2}$. Consequently, there exists an integer which is strictly less than $\frac{n}{2}$ such that when multiplied by $k_{1}$ gives a (non-zero) multiple of $n$, that is, $(\alpha+\beta)n$ or $(\alpha-\beta)n$. However, this is a contradiction, because, from Theorem \ref{theorem iso accordions}, $\gcd(n,k_{1})=2$ and so the least common multiple of $n$ and $k_{1}$ is $\frac{k_{1}n}{2}$. 
\end{proof}

\subsection{Isomorphisms between quartic circulant graphs and accordion graphs}\label{section circulant accordion graphs}

By Theorem 1 in \cite{heuberger}, the quartic circulant graph $\textrm{Ci}[2n,\{a,b\}]$ is bipartite if and only if both $a$ and $b$ are odd. We shall need this result when proving Theorem \ref{theorem bip iso} and Theorem \ref{theorem nonbip iso}.

\subsubsection{Bipartite circulants which are accordion graphs}

\begin{theorem}\label{theorem bip iso}
Let both $a$ and $b$ be odd. Then, $\textrm{Ci}[2n,\{a,b\}]\simeq A[n,k]$ if and only if the following three conditions hold:
\begin{enumerate}[(i)]
\item $n$ is even and $k=2$;
\item $\gcd(2n,a)=\gcd(2n,b)=1$; and
\item $a+b=n$.
\end{enumerate}
\end{theorem}

\begin{proof}
As already stated above, since both $a$ and $b$ are odd, $\textrm{Ci}[2n,\{a,b\}]$ is bipartite. Moreover, by Lemma \ref{lemma bipartite}, $A[n,k]$ is bipartite if and only if both $n$ and $k$ are even. From Theorem \ref{theorem accordion char circ}, we also know that the only bipartite accordion graphs which are circulant are the ones having $k=2$. So it suffices to show that, for $n$ even, a bipartite circulant graph $\textrm{Ci}[2n,\{a,b\}]$ is isomorphic to $A[n,2]$ if and only if $\gcd(2n,a)=\gcd(2n,b)=1$ and $a+b=n$.
From Lemma 5.2 in \cite{accordionspmh}, we already know that $\textrm{Ci}[8,\{1,3\}]\simeq A[4,2]$, and since $1$ and $3$ are the only two distinct odd integers in $\{1,2,3\}$ which satisfy conditions (ii) and (iii) in the statement of Theorem \ref{theorem bip iso}, the result holds for $n=4$. Thus, in what follows we can assume $n\geq 6$.\\

($\Rightarrow$) Assume $\textrm{Ci}[2n,\{a,b\}]\simeq A[n,2]$. In what follows we denote $\textrm{Ci}[2n,\{a,b\}]$, or equivalently $A[n,2]$, by $G$. \\

\noindent\textbf{Claim 1.} $\gcd(2n,a)=\gcd(2n,b)=1$. 

\noindent{\emph{Proof of Claim 1.}} Suppose that $\gcd(2n,a)\not=1$, for contradiction. Then, the least common multiple of $a$ and $2n$ is $2na'$, for some $a'<a$. Consequently, since $a$ is odd, there exists an even integer $p$ such that $ap=2na'$. Moreover, since $a\neq a'$ and $a$ is odd, $\frac{a}{a'}$ (or equivalently $\frac{2n}{p}$) is odd and at least $3$, and so $p<n$. 
By considering the edges in $G$ having length $a$, there exists a partition $\mathcal{P}$ of the $2n$ vertices of $G$ into $\frac{2n}{p}$ sets, each inducing a $p$-cycle. This follows since $\gcd(2n,a)=\frac{2n}{p}\not=1$. We remark that $\mathcal{P}$ has an odd number of components, namely $\gcd(2n,a)$, or equivalently $\frac{2n}{p}$. 

Since $G$ is a connected quartic circulant graph and $\frac{2n}{p}>1$, two vertices on a particular $p$-cycle in $\mathcal{P}$ are adjacent in $G$ if and only if there is an edge of length $a$ between them. In other words, the subgraph induced by the vertices on a $p$-cycle in $\mathcal{P}$ is the $p$-cycle itself. 
Therefore, the graph contains two adjacent vertices $x_{i}$ and $x_{j}$ belonging to two different $p$-cycles from $\mathcal{P}$. Consequently, since the length of the edge $x_{i}x_{j}$ is $b$, $i\equiv j\pm b\pmod{2n}$, and so, the vertices of these two $p$-cycles induce $C_{p}\square P_{2}$, where $P_{2}$ is the path on two vertices. By a similar argument to that used on $x_{i}$ and $x_{j}$, we deduce that $G$ contains a spanning subgraph $G_{0}$ isomorphic to $C_{p}\square P_{\frac{2n}{p}}$.
We next claim that: 
\begin{enumerate}[(a)]
\item given two adjacent vertices from some $p$-cycle in $\mathcal{P}$, say $x_{i}$ and $x_{i+a}$, if $x_{i}$ is a vertex in $\mathcal{U}=\{u_{1},u_{2},\ldots, u_{n}\}$, then $x_{i+a}$ is a vertex in $\mathcal{V}=\{v_{1},v_{2},\ldots, v_{n}\}$, or vice-versa; and
\item given two adjacent vertices from two different $p$-cycles in $\mathcal{P}$, say $x_{i}$ and $x_{i+b}$, we have that either both belong to $\mathcal{U}$, or both belong to $\mathcal{V}$.
\end{enumerate}
First of all, we note that the vertices inducing a $p$-cycle from $\mathcal{P}$, cannot all belong to $\mathcal{U}$, since the latter set of vertices induces a $n$-cycle, and $p<n$. Similarly, the vertices inducing a $p$-cycle from $\mathcal{P}$, cannot all belong to $\mathcal{V}$. Secondly, let $i\in[2n]$, such that $x_{i}$ is of degree $3$ in $G_{0}$, and $x_{i}x_{i+b}\in E(G_{0})$. Consider the $4$-cycle $(x_{i}, x_{i+a},x_{i+a+b}, x_{i+b})$. Since $n>4$, these four vertices cannot all belong to $\mathcal{U}$ (or $\mathcal{V}$). Also, we cannot have three of them belonging to $\mathcal{U}$ (or $\mathcal{V}$). For, suppose not. Without loss of generality, suppose $x_{i}=u_{1}$, $x_{i+a}=u_{2}$, and $x_{i+a+b}=u_{3}$. The vertex $x_{i+b}$ can either be equal to $v_{1}$ or $v_{3}$, but, since $n>4$ and $k=2$, $x_{i+b}$ must be equal to $v_{3}$. Furthermore, the vertex $x_{i+2a}$ can be equal to $v_{2}$ or $v_{4}$. But, $v_{2}v_{3}$ and $v_{3}v_{4}$ are both edges in $A[n,2]$, and so $x_{i+2a}$ is adjacent to $x_{i+b}$, a contradiction, since $G_{0}\simeq C_{p}\square P_{\frac{2n}{p}}$.
By the same reasoning, one can deduce that there is no $4$-cycle in $G_{0}$ having exactly three vertices belonging to $\mathcal{U}$ (or $\mathcal{V}$).

Therefore, exactly two vertices from $(x_{i}, x_{i+a},x_{i+a+b},x_{i+b})$ belong to $\mathcal{U}$, and the other two belong to $\mathcal{V}$. Without loss of generality, assume that $x_{i}$ belongs to $\mathcal{U}$.
Suppose that $x_{i+b}\not\in\mathcal{U}$, for contradiction. Then, $\mathcal{U}$ must contain exactly one of $x_{i+a}$ and $x_{i+a+b}$. Suppose we have $x_{i+a+b}\in\mathcal{U}$. Consequently, $x_{i+a}$ and $x_{i+b}$ belong to $\mathcal{V}$. Without loss of generality, let $x_{i}=u_{1}$. Then, $x_{i+a}$ and $x_{i+b}$ are equal to $v_{1}$ and $v_{3}$. This means that there exists a vertex in $\mathcal{U}-\{u_{1}\}$ which is adjacent to both $v_{1}$ and $v_{3}$. However, this is a contradiction, because when $n\geq 6$, the only vertex in $A[n,2]$ which is adjacent to both $v_{1}$ and $v_{3}$ is $u_{1}$. 
Therefore, we have $x_{i+a}\in\mathcal{U}$. This means that $x_{i+b}$ and $x_{i+a+b}$ both belong to $\mathcal{V}$. Suppose further that $x_{i+2a}\in\mathcal{V}$. Since we cannot have three vertices in a $4$-cycle belonging to $\mathcal{V}$, $x_{i+2a+b}\in \mathcal{U}$. However, this once again gives rise to two vertices in $\mathcal{V}$, say $v_{\beta}$ and $v_{\beta+2}$, for some $\beta\in[n]$, which have two distinct vertices from $\mathcal{U}$ as common neighbours, a contradiction.  
Therefore, $x_{i+2a}$ must belong to $\mathcal{U}$. By repeating the same argument used for $x_{i+2a}$ to the vertices $x_{i+3a},\ldots, x_{i+(p-1)a}$, we get that all the vertices in the $p$-cycle, from $\mathcal{P}$, containing $x_{i}$, belong to $\mathcal{U}$, a contradiction. Hence, $x_{i+a}\not\in\mathcal{U}$, and consequently, neither one of $x_{i+a}$ and $x_{i+a+b}$ is in $\mathcal{U}$, contradicting our initial assumption. This implies that $x_{i+b}\in\mathcal{U}$, and that $x_{i\pm a}$ and $x_{i+b\pm a}$ belong to $\mathcal{V}$. This forces all the vertices not considered so far to satisfy the two conditions (a) and (b) in the above claim.

Thus, by Remark \ref{RemarkDrawing}, $\frac{2n}{p}=\gcd(n,k)$. This is a contradiction, since $\frac{2n}{p}$ is odd and the greatest common divisor of $n$ and $k$ is even, since both $n$ and $k$ are even. Hence, $\gcd(2n,a)=1$, and by a similar reasoning, one can deduce that $\gcd(2n,b)=1$ as well.\,\,\,{\tiny$\blacksquare$}\\

\noindent\textbf{Claim 2.} $a+b=n$. 

\noindent{\emph{Proof of Claim 2.}} Suppose that $a+b\neq n$. By Claim 1, we have that the edges of $G$ can be partitioned into two Hamiltonian cycles induced by the edges having length $a$ and $b$, respectively. This means that there exist two consecutive edges, each with endvertices belonging to $\mathcal{U}$, having unequal lengths, that is, length $a$ and length $b$. Without loss of generality, assume that $u_{n}u_{1}$ and $u_{1}u_{2}$ have lengths $b$ and $a$, respectively. Consider the $4$-cycle $C=(u_{1},u_{2},v_{2},v_{1})$. Since the edges having length $a$ (and similarly the edges having length $b$) induce a Hamiltonian cycle, and $n>4$, the lengths of the edges in $C$ cannot all be the same. Moreover, if $C$ contains exactly two edges having length $a$, then these two edges cannot be consecutive. For, suppose not. This implies that $2a\equiv\pm 2b\pmod{2n}$. Since $a$ and $b$ are both positive and distinct, $2a\equiv -2b\pmod{2n}$, and consequently, $a+b=n$, a contradiction to our initial assumption. By a similar reasoning, if a $4$-cycle in $G$ has exactly two edges of length $b$ they cannot be consecutive. In the sequel, when a $4$-cycle in $G$ has exactly two edges having length $a$ (similarly, length $b$) and these two edges are consecutive, we shall say that this $4$-cycle has the forbidden configuration.

Since $C$ cannot have the forbidden configuration, the lengths of the edges $(u_{1}u_{2}, u_{2}v_{2}, \linebreak v_{2}v_{1}, v_{1}u_{1})$ of $C$ can be of Type A1 $:=(a,b,b,b)$, Type A2 $:=(a,a,b,a)$, Type A3 $:=(a,a,a,b)$, Type A4 $:=(a,b,a,a)$, or Type A5 $:=(a,b,a,b)$, which are depicted in Figure \ref{FigureTypes}. We remark that any other 4-cycle in $G$ must admit exactly 3 edges of length $a$ (or $b$), or be similar to Type A5, that is, has exactly two edges of length $a$ which are not consecutive.

\begin{figure}[h]
\centering
\includegraphics[width=1\textwidth, keepaspectratio]{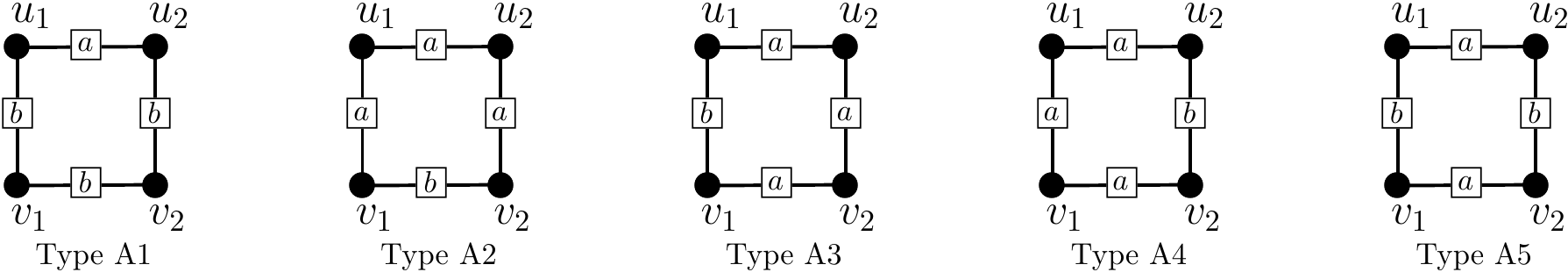}
\caption{Different lengths of the edges in $C$}
\label{FigureTypes}
\end{figure} 

Suppose that $C$ is of Type A1. Then, $u_{1}v_{3}$ and $v_{2}v_{3}$ both have length $a$, implying that the 4-cycle $(u_{1},v_{3},v_{2},v_{1})$ has the forbidden configuration, a contradiction. Therefore, $C$ cannot be of Type A1. Suppose that $C$ is of Type A2. Consequently, $u_{1}v_{3}$ has length $b$. Since the $4$-cycle $(u_{1},v_{3},v_{2},u_{2})$ cannot have the forbidden configuration, $v_{2}v_{3}$ must be equal to $a$. However, this means that $u_{n}v_{2}$ has length $b$, and so, the $4$-cycle $(u_{n},u_{1},u_{2},v_{2})$ has the forbidden configuration, a contradiction. Therefore, $C$ cannot be of Type A2 either. If $C$ is of Type A3, then the edge $u_{n}v_{2}$ is of length $b$, implying that the $4$-cycle $(u_{n},u_{1},u_{2},v_{2})$ has the forbidden configuration, a contradiction once again, implying that $C$ cannot be of Type A3.
Hence, suppose that $C$ is of Type A4. Since the edge $u_{1}v_{3}$ has length $b$, the edge $v_{2}v_{3}$ must have length $a$, because otherwise, the 4-cycle $(u_{1},v_{3},v_{2},v_{1})$ has the forbidden configuration. However, this means that the edge $u_{n}v_{2}$ has length $b$, implying that the 4-cycle $(u_{n},u_{1},v_{1},v_{2})$ has the forbidden configuration. Therefore, $C$ must be of Type A5, and consequently, $u_{1}v_{3}$ has length $a$. The edge $v_{2}v_{3}$ cannot have length $b$, because otherwise, the $4$-cycle $(u_{1},u_{2},v_{2},v_{3})$ has the forbidden configuration. Therefore, $v_{2}v_{3}$ has length $a$, implying that the $4$-cycle $(u_{n},u_{1},v_{3},v_{2})$ has the forbidden configuration, a contradiction once again. Thus, $C$ cannot be of Type A5, and so, $a+b=n$, as required. \,\,\,{\tiny$\blacksquare$}\\

($\Leftarrow$) Assume that $\gcd(2n,a)=\gcd(2n,b)=1$ and $a+b=n$. We already know from Lemma 5.2 in \cite{accordionspmh} that $A[n,2]\simeq\textrm{Ci}[2n,\{1,n-1\}]$. So it suffices to show that when $a$ and $b$ are not equal to $1$ and $n-1$, respectively, $\textrm{Ci}[2n,\{a,b\}]\simeq\textrm{Ci}[2n,\{1,n-1\}]$. Let $V(\textrm{Ci}[2n,\{1,n-1\}])=\{x_{1},\ldots, x_{2n}\}$ and let $V(\textrm{Ci}[2n,\{a,b\}])=\{x_{1}',\ldots, x_{2n}'\}$. We claim that the function $\mu:V(\textrm{Ci}[2n,\{1,n-1\}])\rightarrow V(\textrm{Ci}[2n,\{a,b\}])$ defined by $\mu:x_{i}\mapsto x_{ia}$ for all $1\leq i \leq 2n$, is an isomorphism. Since $\gcd(2n,a)=1$, the Hamiltonian cycle $(x_{1},x_{2},\ldots, x_{2n})$ of $\textrm{Ci}[2n,\{1,n-1\}]$ is mapped to the Hamiltonian cycle $(x_{a}',x_{2a}',\ldots,x_{2na}')$ of $\textrm{Ci}[2n,\{a,b\}]$. Consequently, $\mu$ is bijective, and the edges having length $1$ in $\textrm{Ci}[2n,\{1,n-1\}]$ are mapped to the edges having length $a$ in $\textrm{Ci}[2n,\{a,b\}])$. What is left to show is that an edge having length $n-1$ in $\textrm{Ci}[2n,\{1,n-1\}]$ is mapped to an edge of length $b$ in $\textrm{Ci}[2n,\{a,b\}]$. Let $j\in[2n]$ and consider the edge $x_{j}x_{j+n-1}$ from $\textrm{Ci}[2n,\{1,n-1\}]$. By definition of $\mu$, $\mu(x_{j})=x_{ja}'$ and $\mu(x_{j+n-1})=x_{(j+n-1)a}'$. Since $\gcd(2n,a)=1$, $a$ is odd. Moreover, since $a+b=n$, we have $(j+n-1)a\equiv ja+n-a\equiv ja+b\pmod{2n}$. Therefore, $\mu(x_{j})\mu(x_{j+n-1})\in E(\textrm{Ci}[2n,\{a,b\}])$, and consequently, edges having length $n-1$ in $\textrm{Ci}[2n,\{1,n-1\}]$ are mapped to edges having length $b$ in $\textrm{Ci}[2n,\{a,b\}]$. Since $\textrm{Ci}[2n,\{1,n-1\}]$ and $\textrm{Ci}[2n,\{a,b\}]$ have the same number of edges, $\textrm{Ci}[2n,\{a,b\}]$ is isomorphic to $\textrm{Ci}[2n,\{1,n-1\}]$, as required.
\end{proof}

\subsubsection{Non-bipartite circulants which are accordion graphs}

Before continuing, we remark that by Proposition 1 in \cite{boeschtindell}, the quartic circulant graph $\textrm{Ci}[2n,\{a,b\}]$ is connected if and only if $\gcd(2n,a,b)=1$. Consequently, if $\textrm{Ci}[2n,\{a,b\}]$ is connected, $a$ and $b$ cannot both be even, because otherwise, $\gcd(2n,a,b)\geq 2$. Thus, for the non-bipartite case, $a$ and $b$ need to have different parity. In the next theorem, without loss of generality, we shall assume that $a$ is odd and $b$ is even.

\begin{theorem}\label{theorem nonbip iso}
Let $a$ be an odd integer and $b$ be an even integer. Then, $\textrm{Ci}[2n,\{a,b\}]\simeq A[n,k]$ if and only if there exists $1\leq k\leq\frac{n}{2}$ such that the following three conditions hold:
\begin{enumerate}[(i)]
\item $k$ is odd when $n$ is even;
\item $\gcd(2n,a)=\gcd(n,k)$; and 
\item $b\gcd(n,k)\equiv \pm2\lambda a\pmod{2n}$, where $\lambda$ is the smallest integer such that $\lambda k\equiv \gcd(n,k)\pmod{n}$. 
\end{enumerate}
\end{theorem}

\begin{proof}

When $n=3$, the only possible value for $k$ is $1$. From Corollary 5.5 in \cite{accordionspmh}, we already know that $A[3,1]\simeq \textrm{Ci}[6,\{1,2\}]$, and since $1$ and $2$ are the only two distinct integers (with different parity) in $\{1,2\}$ satisfying conditions (ii) and (iii) in the statement of Theorem \ref{theorem nonbip iso}, the result holds for $n=3$. Thus, in what follows we can assume that $n\geq 4$.\\

($\Rightarrow$) Since $a$ and $b$ have different parity, $\textrm{Ci}[2n,{a,b}]$ is not bipartite. Thus, by Lemma \ref{lemma bipartite}, if $n$ is even then $k$ must be odd. This proves (i). As in Theorem \ref{theorem bip iso}, we denote $\textrm{Ci}[2n,\{a,b\}]$, or equivalently $A[n,k]$, by $G$.\\

\noindent\textbf{Claim A.} $\gcd(2n,a)=\gcd(n,k)$.
 
\noindent Let $\gcd(2n,a)=q$ and suppose that $q\neq\gcd(n,k)$, for contradiction. Since $a$ is odd, $q$ is an odd number as it is the greatest common divisor of two integers having distinct parity. The least common multiple of $2n$ and $a$ is equal to $2na'$, for some $a'\leq a$, and consequently, there exists an even integer $p$ (equal to $\frac{2n}{q}$) such that $ap=2na'$. By considering the edges in $G$ having length $a$, there exists a partition $\mathcal{P}$ of the $2n$ vertices of $G$ into $\frac{2n}{p}$ sets, each inducing a $p$-cycle. Moreover, $\mathcal{P}$ has an odd number of components, namely $q$, or equivalently $\frac{2n}{p}$. We consider two cases depending on the number of components of $\mathcal{P}$, that is, when $q=1$ and when $q\geq 3$.\\

\noindent\textbf{Case 1.} $q=1$.

\noindent Since $q=\gcd(2n,a)$, the edges having length $a$ induce a Hamiltonian cycle of $G$. When $n=4$, since $G$ is not bipartite, $k$ must be equal to $1$. However, $\gcd(4,1)=1$, and so, since we are assuming that $\gcd(n,k)\neq q$, we can assume that $n>4$. Consider a $4$-cycle in $G$. Since $n>4$, we cannot have all the edges of a $4$-cycle having length $a$. Moreover, since $a$ is odd and $b$ is even, we cannot have $a\equiv\pm3b\pmod{2n}$ or $b\equiv\pm3a\pmod{2n}$. Therefore, the length of the consecutive edges of a $4$-cycle in $G$ must be $(a,a,b,b)$ or $(a,b,a,b)$, as in Figure \ref{FigureTypes2}.

\begin{figure}[h]
\centering
\includegraphics[width=0.35\textwidth, keepaspectratio]{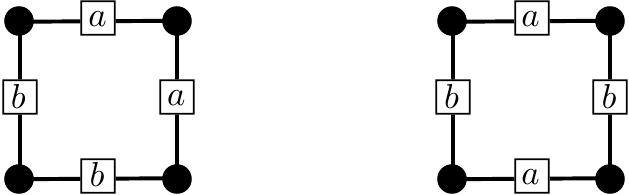}
\caption{Edge lengths of $4$-cycles in Case 1 of Theorem \ref{theorem nonbip iso}}
\label{FigureTypes2}
\end{figure}
In particular, since the edges having length $a$ induce a Hamiltonian cycle of $G$, there must exist a $4$-cycle $(u_{i}, u_{i+1},v_{i+1},v_{i})$, for some $i\in[n]$, in $G$, such that its consecutive edges have lengths $(a,a,b,b)$. Consequently, $2a\equiv\pm 2b\pmod{2n}$, and since $a\neq b$, we have $2a\equiv -2b\pmod{2n}$. As a result, $a+b=n$, and so, since $a$ is odd and $b$ is even, $n$ is odd. We claim that, in this case, $\gcd(2n,b)=2$. In fact, since $2n=2(n-a)+2a$, $\gcd(2n,b)=\gcd(2n,n-a)=\gcd(2n,2a)=2\gcd(n,a)$. Moreover, since $a$ is odd, $\gcd(2n,a)=\gcd(n,a)$, giving $\gcd(2n,b)=2$, since $\gcd(2n,a)=1$. Consequently, by Remark \ref{RemarkDrawing}, $G$ is isomorphic to some accordion graph $A[n,k_{1}]$, with $\gcd(n,k_{1})=1$, in which the spokes form a Hamiltonian cycle, and the cycle edges have length $b$. By Theorem \ref{theorem iso accordions}, $k=k_{1}$, implying that $\gcd(n,k)$ is equal to $1$, contradicting our initial assumption that $\gcd(n,k)\neq q$. Therefore, $\gcd(2n,a)=\gcd(n,k)$.\\

\noindent\textbf{Case 2.} $q\geq 3$.

\noindent Since $G$ is a connected quartic graph and $q>1$, two vertices on a particular $p$-cycle in $\mathcal{P}$ are adjacent in $G$ if and only if there is an edge of length $a$ between them, in other words, the subgraph induced by the vertices on a $p$-cycle in $\mathcal{P}$ is the $p$-cycle itself. 
Therefore, the graph contains two adjacent vertices $x_{i}$ and $x_{j}$ belonging to two different $p$-cycles from $\mathcal{P}$, with the edge $x_{i}x_{j}$ having length $b$. Consequently, since $i\equiv j\pm b\pmod{2n}$, the vertices of these two $p$-cycles induce $C_{p}\square P_{2}$, where $P_{2}$ is the path on two vertices. By a similar argument to that used on $x_{i}$ and $x_{j}$, we deduce that $G$ contains a spanning subgraph $G_{0}$ isomorphic to $C_{p}\square P_{q}$.
As in Theorem \ref{theorem bip iso} we next claim that: 
\begin{enumerate}[(a)]
\item given two adjacent vertices from some $p$-cycle in $\mathcal{P}$, say $x_{i}$ and $x_{i+a}$, if $x_{i}$ is a vertex in $\mathcal{U}$, then $x_{i+a}$ is a vertex in $\mathcal{V}$, or vice-versa; and
\item given two adjacent vertices from two different $p$-cycles in $\mathcal{P}$, say $x_{i}$ and $x_{i+b}$, we have that either both belong to $\mathcal{U}$ or both belong to $\mathcal{V}$.
\end{enumerate}
First of all, we note that the vertices inducing a $p$-cycle from $\mathcal{P}$, cannot all belong to $\mathcal{U}$, since the latter set of vertices induces a $n$-cycle, and $p<n$. Similarly, the vertices inducing a $p$-cycle from $\mathcal{P}$, cannot all belong to $\mathcal{V}$. Secondly, let $i\in[2n]$, such that $x_{i}$ is of degree $3$ in $G_{0}$, and $x_{i}x_{i+b}\in E(G_{0})$. Consider the $4$-cycle $(x_{i}, x_{i+a},x_{i+a+b}, x_{i+b})$. If these four vertices all belong to $\mathcal{U}$ (or $\mathcal{V}$), then $n=4$. However, there is no odd integer $a$ such that $\gcd(8,a)=q>1$, and so we can assume that these four vertices cannot all belong to $\mathcal{U}$ (or $\mathcal{V}$).
Therefore, suppose that exactly three of them belong to $\mathcal{U}$ (or $\mathcal{V}$). Without loss of generality, assume that $x_{i}=u_{1}$, $x_{i+a}=u_{2}$, and $x_{i+a+b}=u_{3}$. We consider two cases, depending whether $x_{i+b}$ is equal to $v_{1}$ or $v_{3}$.\\

\noindent\textbf{Case 2.1} $x_{i+b}=v_{1}$. 

\noindent Since $u_{3}$ is adjacent to $v_{1}$, we have that $k=n-2$, which by Definition \ref{def AccordionGraphs} is at most $\frac{n}{2}$. This implies that $n\leq 4$, and since we are assuming that $n>3$, we must have $n=4$, and $k=1$ or $2$. Since $G$ is not bipartite, $k=1$. In this case, the vertex $x_{i+2a}$ is equal to $v_{2}$ or $v_{3}$. However, this means that $x_{i+2a}$ is adjacent to $x_{i+b}$ or $x_{i+a+b}$, a contradiction, since $G_{0}\simeq C_{p}\square P_{q}$ and $q\geq 3$.\\

\noindent\textbf{Case 2.2} $x_{i+b}=v_{3}$.

\noindent Since $u_{1}$ is adjacent to $v_{3}$, the value of $k$ is $2$. Consequently, $x_{i+2a}$ is equal to $v_{2}$ or $v_{4}$, and in both cases $x_{i+2a}$ is adjacent to $x_{i+b}$, a contradiction once again, as in Case 2.1.\\

We remark that by the same reasoning used in Case 2.1 and Case 2.2 above, no $4$-cycle in $G_{0}$ has exactly three of its vertices belonging to $\mathcal{U}$ (or $\mathcal{V}$). This means that exactly two vertices from $(x_{i}, x_{i+a},x_{i+a+b}, x_{i+b})$ belong to $\mathcal{U}$, and the other two belong to $\mathcal{V}$.  Without loss of generality, assume that $x_{i}$ belongs to $\mathcal{U}$.
Suppose that $x_{i+b}\not\in\mathcal{U}$, for contradiction. Then, $\mathcal{U}$ must contain exactly one of $x_{i+a}$ and $x_{i+a+b}$. Suppose we have $x_{i+a+b}\in\mathcal{U}$. Consequently, $x_{i+a}$ and $x_{i+b}$ belong to $\mathcal{V}$, and so since $q\geq 3$, $x_{i+2a+b}$ and $x_{i+a+2b}$ belong to $\mathcal{U}$. However, this means that the $4$-cycle $(x_{i+a+b},x_{i+a+2b},x_{i+2a+2b},x_{i+2a+b})$ in $G_{0}$ has at least three of its vertices belonging to $\mathcal{U}$, a contradiction. 
Therefore, we have $x_{i+a}\in\mathcal{U}$. This means that $x_{i+b}$ and $x_{i+a+b}$ both belong to $\mathcal{V}$. Suppose further that $x_{i+2a}\in\mathcal{V}$. Since the 4-cycle $(x_{i+a}, x_{i+a+b}, x_{i+2a+b},x_{i+2a},)$ cannot have three vertices belonging to $\mathcal{V}$, $x_{i+2a+b}\in \mathcal{U}$. However, since $x_{i+a+b}$ and $x_{i+2a}$ belong to $\mathcal{V}$, the $4$-cycle $(x_{i+2a+b},x_{i+2a+2b},x_{i+3a+2b}, x_{i+3a+b})$ in $G_{0}$ has at least three of its vertices belonging to $\mathcal{U}$, a contradiction. Therefore, $x_{i+2a}$ must belong to $\mathcal{U}$. By repeating the same argument used for $x_{i+2a}$ to the vertices $x_{i+3a},\ldots, x_{i+(p-1)a}$, we get that all the vertices in the $p$-cycle (from $\mathcal{P}$) containing $x_{i}$, belong to $\mathcal{U}$, a contradiction. Hence, $x_{i+a}\not\in\mathcal{U}$. Thus, $x_{i+a}$ and $x_{i+a+b}$ do not belong to $\mathcal{U}$, contradicting our initial assumption. This implies that $x_{i+b}\in\mathcal{U}$, and that $x_{i\pm a}$ and $x_{i+b\pm a}$ belong to $\mathcal{V}$. This forces all the vertices not considered so far to satisfy the two conditions in the above claim.
Consequently, by Remark \ref{RemarkDrawing}, $G$ is isomorphic to some accordion graph $A[n,k_{2}]$, with $\gcd(n,k_{2})=q\geq 3$. By Theorem \ref{theorem iso accordions}, $k=k_{2}$, implying that $\gcd(n,k)$ is equal to $q$, contradicting our initial assumption that $\gcd(n,k)\neq q$. This proves (ii) in the statement of Theorem \ref{theorem nonbip iso}.\,\,\,{\tiny$\blacksquare$}\\

\noindent\textbf{Claim B.} $b\gcd(n,k)\equiv \pm2\lambda a\pmod{2n}$, where $\lambda$ is the smallest integer such that $\lambda k\equiv \gcd(n,k)\pmod{n}$.
 
\noindent Let $q=\gcd(n,k)$ and let $p'=\frac{p}{2}$, that is, $p'=\frac{n}{q}$. For every $i\in[q]$, let 
\begin{linenomath}$$W_{i}=\{v_{i},u_{i}, v_{i+k}, u_{i+k},\ldots, v_{i+(p'-1)k},u_{i+(p'-1)k}\}.$$
\end{linenomath}
By Remark \ref{RemarkDrawing} and Claim A, the sets $W_{1},\ldots, W_{q}$ are the $q$ vertex sets of the $p$-cycles in the partition $\mathcal{P}$, induced by the edges of length $a$ in $G$. Since for every $i\in[n]$, the length of the edges $u_{i}u_{i+1}$ and $v_{i}v_{i+1}$ is not $a$, the $n$-cycles $(u_{1},\ldots,u_{n})$ and $(v_{1},\ldots,v_{n})$ must be generated by edges having length $b$. 
Next, we define the $[v_{1},v_{q+1}]_b$-path to be the path containing $v_{2}$, having $v_{1}$ and $v_{q+1}$ as endvertices, and whose edges are all of length $b$. Without loss of generality, assume that $v_{1}$ and $v_{2}$ are respectively equal to $x_{j}$ and $x_{j+b}$, for some $j\in[2n]$. Consequently, $x_{j+bq}=v_{q+1}$, and we shall say that the length of this $[v_{1},v_{q+1}]_b$-path is the sum of all the edge lengths on this path. In this case, the length of the $[v_{1},v_{q+1}]_b$-path is $bq$.  By Remark \ref{RemarkDrawing}, both $v_{1}$ and $v_{q+1}$ belong to $W_{1}$, that is, to the same $p$-cycle in $\mathcal{P}$. Similarly, we define the $[v_{1},v_{q+1}]_a$-path to be the path containing $u_{1}$, having $v_{1}$ and $v_{q+1}$ as endvertices, and whose edges are all of length $a$. The length of the $[v_{1},v_{q+1}]_a$-path is $2a$ multiplied by $\lambda$, where $\lambda$ is the least integer such that $\lambda k\equiv q\pmod{n}$. As a consequence, since $v_{1}=x_{j}$, we have $v_{q+1}=x_{j\pm 2a\lambda}$. The reason behind the $\pm$ sign preceding $2a$ is because we have already assumed that $v_{1}$ and $v_{2}$ are equal to $x_{j}$ and $x_{j+b}$, respectively, and so now we need to take into consideration that $u_{1}$ can be equal to $x_{j+a}$ or $x_{j-a}$. Furthermore, since $v_{q+1}$ is also equal to $x_{j+bq}$ and $q=\gcd(n,k)$, $b\gcd(n,k)\equiv \pm2\lambda a\pmod{2n}$, proving (iii).\,\,\,{\tiny$\blacksquare$}\\

($\Leftarrow$) Assume there exists $1\leq k\leq \frac{n}{2}$ satisfying (i), (ii), (iii) in the statement of Theorem \ref{theorem nonbip iso} and let $\gcd(n,k)=q=\frac{2n}{p}$, for some integers $p$ and $q$. Since $\gcd(2n,a)=\gcd(n,k)$, both $\textrm{Ci}[2n,\{a,b\}]$ and $A[n,k]$ contain $C_{p}\square P_{q}$ as a spanning subgraph. We note that since $n$ and $k$ do not have the same parity, $q$ is odd and $p$ is even. Let $G_{\textrm{\tiny A}}$ and $G_{\textrm{\tiny Ci}}$ be the copy of a spanning subgraph $C_{p}\square P_{q}$ in $A[n,k]$ and $\textrm{Ci}[2n,\{a,b\}]$, respectively. Consider the graph $G_{\textrm{\tiny Ci}}$. Since (quartic) circulant graphs are vertex-transitive, without loss of generality, we can assume that the vertex $x_{a+b}$ has degree $3$ in $G_{\textrm{\tiny Ci}}$. We first consider the case when $bq\equiv 2\lambda a\pmod{2n}$. For every $i\in [q]$ let
\begin{linenomath}$$X_{i}=\{x_{a+ib},x_{2a+ib},\ldots,x_{pa+ib}\}.$$
\end{linenomath}
It is clear that each $X_{i}$ induces a $p$-cycle, and that each $p$-cycle is induced by edges of $\textrm{Ci}[2n,\{a,b\}]$ having length $a$. 
The $X_{i}$s are in fact the $q$ canonical copies of $C_{p}$ in $G_{\textrm{\tiny Ci}}$.
Next, consider the graph $G_{\textrm{\tiny A}}$. By Remark \ref{RemarkDrawing}, we can assume that $v_{1}$ has degree $3$ in $G_{\textrm{\tiny A}}$. For every $i\in [q]$ let
\begin{linenomath}$$W_{i}=\{v_{i},u_{i}, v_{i+k}, u_{i+k},\ldots, v_{i+(p'-1)k},u_{i+(p'-1)k}\},$$
\end{linenomath}
where $p'=\frac{p}{2}$. It is clear that each $W_{i}$ induces a $p$-cycle, and the $W_{i}$s are in fact the $q$ canonical copies of $C_{p}$ in $G_{\textrm{\tiny A}}$. Consider the function $\phi:V(\textrm{Ci}[2n,\{a,b\}])\rightarrow V(A[n,k])$ such that, for every $i\in[q]$, $\phi$ maps the consecutive vertices of the $p$-cycle induced by $X_{i}$ to the consecutive vertices of the $p$-cycle induced by $W_{i}$, such that $\phi(x_{a+ib})=v_{i}$ and $\phi(x_{2a+ib})=u_{i}$. Since $\phi$ induces an isomorphism between $G_{\textrm{\tiny Ci}}$ and $G_{\textrm{\tiny A}}$, $\phi$ is clearly bijective. Moreover, since $v_{q}v_{q+1}\in A[n,k]$, in order to show that $\phi$ is an isomorphism between $\textrm{Ci}[2n,\{a,b\}]$ and $A[n,k]$, it suffices to show that $\phi^{-1}(v_{q})$ and $\phi^{-1}(v_{q+1})$ are adjacent in $\textrm{Ci}[2n,\{a,b\}])$. By definition of $X_{q}$, $\phi^{-1}(v_{q})=x_{a+qb}$, where the latter is of degree 3 in $G_{\textrm{\tiny Ci}}$. The vertex $x_{a+b+qb}$ is the vertex in $X_{1}$ which is adjacent to $x_{a+qb}$ in $\textrm{Ci}[2n,\{a,b\}]$. We claim that $\phi(x_{a+b+qb})=v_{q+1}$. Since $bq\equiv 2\lambda a\pmod{2n}$, $x_{a+b+qb}=x_{a+b+2\lambda a}$, and by definition of $X_{1}$ and $W_{1}$, $\phi(x_{a+b+2\lambda a})=v_{\gamma}$, for some $\gamma\in[n]$. Since $\phi(x_{a+b})=v_{1}$ and $\lambda$ is the smallest integer such that $\lambda k\equiv q\pmod{n}$, $\gamma\equiv 1+\lambda k\equiv 1+q\pmod{n}$, and so $v_{\gamma}=v_{1+q}$, as required. 

Finally, let's consider the case when $bq\equiv -2\lambda a\pmod{2n}$. For every $i\in[q]$, we let
\begin{linenomath}$$X_{i}=\{x_{a+ib},x_{ib},\ldots,x_{(2-p)a+ib}\},$$
\end{linenomath}
and \begin{linenomath}$$W_{i}=\{v_{i},u_{i}, v_{i+k}, u_{i+k},\ldots, v_{i+(p'-1)k},u_{i+(p'-1)k}\},$$
\end{linenomath}
where once again, $p'=\frac{p}{2}$. As in the case when $bq\equiv 2\lambda a\pmod{2n}$, the function $\Phi:V(\textrm{Ci}[2n,\{a,b\}])\rightarrow V(A[n,k])$ which, for every $i\in[q]$, maps the consecutive vertices of the $p$-cycle induced by $X_{i}$ to the consecutive vertices of the $p$-cycle induced by $W_{i}$, such that $\Phi(x_{a+ib})=v_{i}$ and $\Phi(x_{ib})=u_{i}$, is an isomorphism. In particular, $\Phi(x_{a+qb})=v_{q}$ and $\Phi(x_{a+(q+1)b})=v_{q+1}$, because $\Phi(x_{a+(q+1)b})=\Phi(x_{a+b+qb})=\Phi(x_{a+b-2\lambda a})=v_{1+\lambda k}=v_{1+q}$, since  $\lambda$ is the smallest integer such that $\lambda k\equiv q\pmod{n}$. 
\end{proof}

\end{document}